\documentclass[12pt]{amsart}

\usepackage{graphicx}
\usepackage{amssymb}
\usepackage{amsmath, amscd}
\usepackage{hyperref}
\usepackage{wrapfig, framed, caption}

\newtheorem{thm}{Theorem}[section]

\newtheorem{theorem}[thm]{Theorem}
\newtheorem{proposition}[thm]{Proposition}
\newtheorem{corollary}[thm]{Corollary}
\newtheorem{lemma}[thm]{Lemma}

\theoremstyle{definition}
\newtheorem{definition}[thm]{Definition}

\theoremstyle{remark}
\newtheorem{remark}[thm]{Remark}

\newcommand{\R}{{\mathbb{R}}}

\newcommand{\sk}{\operatorname{Sk}}

\newcommand{\bT}{\mathbf{T}}

\usepackage[margin=1in,marginparwidth=0.8in, marginparsep=0.1in]{geometry}

\usepackage{tikz}
\usetikzlibrary{decorations.markings, shapes.geometric, patterns} 
\usetikzlibrary{decorations.pathmorphing}

\usepackage{epigraph}

\begin{document}
\thanks{TE is supported by the Knut and Alice Wallenberg Foundation and the Swedish Research Council. \\ \indent VS is partially supported by the NSF grant CAREER DMS-1654545.}
\title[colored HOMFLYPT counts holomorphic curves]{colored HOMFLYPT counts holomorphic curves} 
\author{Tobias Ekholm}
\author{Vivek Shende}

\maketitle

\begin{abstract}
We compute the contribution of all multiple covers of an isolated rigid embedded holomorphic annulus, stretching between Lagrangians, to the skein-valued count of open holomorphic curves in a Calabi-Yau 3-fold.
The result agrees with the predictions from topological string theory and we use it to prove the 
Ooguri-Vafa formula \cite{OV} that identifies 
the colored HOMFLYPT invariants of a link with a count of holomorphic curves ending on 
the conormal Lagrangian of the link in the resolved conifold. This generalizes our previous work \cite{SOB} which proved the result for the fundamental color. 
\end{abstract}

\vspace{4mm}


%
%


\section{Introduction}

\thispagestyle{empty}

In \cite{OV}, extending the ideas of \cite{Witten, GV-geometry}, Ooguri and Vafa made a remarkable prediction: the colored
HOMFLYPT invariant of a link in the 3-sphere is the count of all holomorphic curves in the resolved conifold with boundary on the  shifted Lagrangian conormal of the link.  
This prediction both implied  novel structural features of the knot invariants, 
and opened the door to counting  holomorphic curves by manipulating knot invariants 
\cite{Aganagic-Vafa, Aganagic-Marino-Vafa, AKMV}; i.e., counting solutions to the Cauchy-Riemann equation -- a nonlinear PDE -- by solving a problem in combinatorial representation theory.  

Much mathematical evidence supports this picture, 
e.g. \cite{Liu, Katz-Liu, OP, LLLZ,  OS, DSV, M}.  
However, the main ingredient   -- a deformation invariant
open Gromov-Witten theory -- was missing.  Indeed, deformation invariance of Gromov-Witten curve counts amounts, naively, to the following assertion: for a generic 1-parameter family of data, the boundary of the parameterized solution space 
sits entirely above the boundary of the family.  This is well-known to be false for moduli of curves with boundary, which have additional codimension one boundary components
associated to hyperbolic and elliptic boundary bubbling. 

In \cite{SOB} we showed: at these extra 
boundaries
in moduli
the boundaries of the curves themselves look exactly like the terms in the framed skein relation (Figure \ref{skein}). 
Recall that the framed skein module $\sk(L)$ is
the free module generated by framed links in the (three-dimensional) Lagrangian $L$ modulo the skein relations. 
If we count curves by the isotopy class of their boundary in $\sk(L)$, then the result is invariant \cite{SOB}. In the same
article, we also proved the Ooguri-Vafa formula in the simplest case, showing that the (uncolored) HOMFLYPT counts curves in the simplest homology class.

The full Ooguri-Vafa prediction amounts to one identity
per integer partition.  
Here, the integer partitions index
on the one hand the possible windings of the boundary of the holomorphic curve around the longitude in conormal Lagrangian
(a solid torus), and on the other hand the possible `colors' of the colored HOMFLYPT polynomial.  
In this language, the case treated in \cite{SOB} corresponds to the partition ``1=1''. Here we will treat general partitions.  

To explain what this requires,
let us sketch the argument from \cite{SOB}. 
Given a link $K \subset S^3$, one forms the conormal $L_{K} \subset T^* S^3$ and shifts it off the zero section.  For an 
appropriate complex structure, any non-constant holomorphic curve in $T^{\ast}S^{3}$ with boundary on $S^{3}\cup L_{K}$ is a cover of the holomorphic annuli with boundaries tracing the path of the original link as 
the conormal is shifted off.  Focus attention on the lowest degree term, i.e., the embedded annulus.  There is exactly 
one per component of the link.  Counting in the skein, 
this gives the class $\langle K \rangle \subset \sk(S^3)$, tensored with some longitude factors in the skeins of the 
conormal Lagrangians recording the fact that the map went only once around each component of the link.  Perform the conifold transition by first stretching around the zero section of $T^*S^3$. 
 Because $S^* S^3$ has no index zero Reeb orbits, in the stretching all holomorphic curves must end up 
in $T^{\ast}S^{3}\setminus S^{3}$ for sufficiently stretched complex structures. (These may later be identified with curves
in the conifold.)
We count curves by their boundaries, so these 
are counted in the skein as 
some $C_K \cdot \langle \emptyset \rangle \in \sk(S^3)$, where $C_K$ just counts the curves which wind once around each 
component of $K$.    
By invariance, $C_K \cdot \langle \emptyset \rangle = \langle K \rangle$, so $C_K$ is nothing other than the HOMFLYPT polynomial
of $K$.  On the other hand, the curves counted by $C_K$ persist to the conifold, so the count is valid there as well.

\begin{wrapfigure}{r}{\dimexpr 6cm + 2\FrameSep + 2\FrameRule\relax}
\begin{center}
\begin{framed}\raggedleft
\includegraphics[scale=0.2]{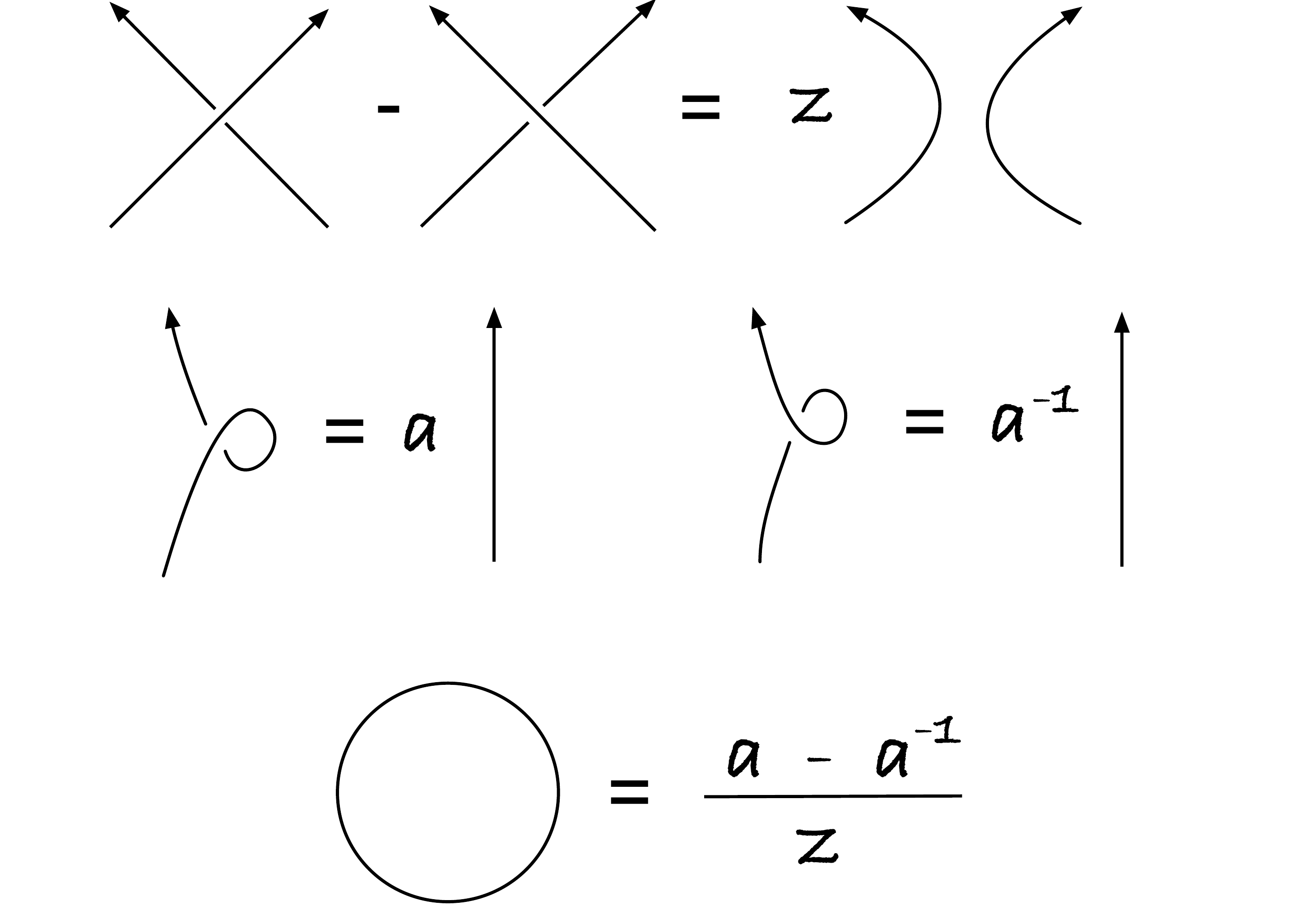}
\caption{\label{skein} The HOMFLYPT skein relations.  Here we will take $z = q^{1/2} - q^{-1/2}$. } 
\end{framed}
\end{center}
\end{wrapfigure}

To treat the general case, we need to solve two problems, one foundational and one calculational.  The foundational problem 
is the usual difficulty of achieving transversality in settings where multiple
covers may appear (which we avoided in \cite{SOB} by considering only the curves going once around), complicated by
our need to {\em not} perturb constant curves.  This is dealt with 
in \cite{bare}.   Here we treat the calculational question: how does the multiply covered holomorphic annulus contribute to the skein?  

Let us first recall how the answer is predicted by \cite{OV}.  The boundary of the worldsheet of a string in a Lagrangian brane  
introduces a line defect into the Chern-Simons theory on the brane.  When the string worldsheet is
embedded it can
be shown to contribute the simplest Wilson line: trace of the holonomy in the fundamental representation \cite[Section 4.4]{Witten}.  
Thus, going $n$ times around should be the trace of the $n^{\rm th}$ power of the holonomy. In the case under consideration we have $n$-fold covers of annuli which come with a $1/n$ automorphism factor.  
Now, exponentiating the sum of these contributions 
(and using a standard identity from symmetric function combinatorics) one finds that the total 
contribution of all disconnected curves should be $\sum_\lambda W_\lambda \otimes W_\lambda$, where $W_\lambda$ 
denotes the Wilson line given by taking trace in the representation corresponding to the partition $\lambda$. 

For various reasons 
it is hard to make direct mathematical use of the above argument:
neither Chern-Simons theory nor topological string theory have proper mathematical foundations,
which makes it hard to interpret the statement that worldsheets in the latter contribute Wilson lines to 
the former.  
Instead in \cite{SOB} we turned the correspondence between embedded curves and 
simple Wilson lines into a definition -- count holomorphic curves by 
their boundaries in the skein -- which we showed was consistent (in full generality, this requires \cite{bare}).  

This definition still does not allow for a direct import of the argument from  \cite{OV}: we 
demand that all curves have embedded boundary, 
which means we can not deal directly with the multiply covered annulus. 
Furthermore, the skein relation requires us to count all disconnected curves, which means we can only see the exponentiated version of the count discussed above. Nevertheless, the final answer given by the Ooguri-Vafa argument is intelligible in our treatment and we can try and verify it.  

Our version of $W_\lambda$ is the element in the skein of the solid torus with the following property: given a framed knot $K$, 
the ordinary HOMFLYPT polynomial of the cable $W_\lambda(K)$ is the same as the $\lambda$-colored HOMFLYPT
polynomial of $K$.  These elements have a skein theoretic characterization as follows. 
The skein of the solid torus has an endomorphism $P_{1,0}$ given by encircling by a meridional loop. 
The $W_\lambda$ are characterized up to scalar multiple 
as the eigenvectors of this endomorphism with eigenvalues 
\cite{Lukac}
\begin{equation}
\label{eigenvalue}
P_{1, 0} W_\lambda  =   (\bigcirc + a (q^{1/2} - q^{-1/2}) c_\lambda(q))  \cdot W_\lambda \\
\end{equation}
Here $\bigcirc$ means an unknot (the boundary of an embedded disk with standard framing), or in other words simply the scalar 
$(a-a^{-1})(q^{1/2} - q^{-1/2})^{-1}$.  The quantity $c_\lambda(q)$ is the `content polynomial' of $\lambda$; 
its value is irrelevant to us here save only for the fact that it determines $\lambda$. 
The scalar multiple is fixed by the quantum dimension formula of \cite{Reshetikhin}, which we recall below
in Equation (\ref{dimension}). 

The $W_\lambda$ span the part of the skein generated by links winding
only positively along the longitude.  Our links will always be positive, being small perturbations 
of positive multiple covers of the longitude.  

We prove here the following version of the Ooguri-Vafa local calculation: 

\begin{theorem} \label{maintheorem}
Let $C$ be a totally isolated rigid holomorphic annulus, with boundaries $K_1, K_2$.  
Then the total contribution of $C$ to the skein valued curve count is 
$$\sum_\lambda \gamma^{|\lambda|} W_\lambda(K_1) \otimes W_\lambda(K_2)$$
This is a sum over all integer partitions $\lambda$, where $|\lambda|$ is the sum of parts of $\lambda$,  and $\gamma$ is some signed
monomial in the framing variables.  
\end{theorem} 

The notion of totally isolated rigid holomorphic annulus (see Definition \ref{standard annuli} below) 
is a  
condition ensuring that the holomorphic annulus is modeled on the kind we
discuss above, and moreover that we may consider its perturbations independently of perturbations needed for
other curves.  The condition is certainly satisfied in the Ooguri-Vafa situation, when
the annulus is the only holomorphic curve around. 

Let us sketch the proof of Theorem \ref{maintheorem}.  By locality of the perturbation scheme
we may work in some appropriate local model.  Perhaps the simplest imaginable is the cotangent
bundle of the solid torus, where we take the zero section and a shifted off conormal of the longitude.  
Studying boundaries of one-dimensional moduli
of curves with one positive puncture as in \cite{AENV, Eicm, ENg, unknot}, we find a relation: 
$((P_{1,0} - \bigcirc) \otimes a_2 - a_1 \otimes (P_{1,0}- \bigcirc)) \Psi = 0$, 
where $\Psi$ is the total contribution of all (perturbed) multiple covers of the  holomorphic annulus.   
Here $a_1, a_2$ are the `$a$' elements of the skeins of
the two solid tori.  From the formula (\ref{eigenvalue}), it follows immediately that $\Psi$ is itself `diagonal' in the eigenbasis, 
i.e., a sum of $W_\lambda \otimes W_\lambda$ with some coefficients.  To determine the coefficients
it suffices to study the unknot, which we did in \cite{unknot}. 

Combining Theorem \ref{maintheorem} with the results of \cite{SOB, bare}, we deduce

\begin{theorem} \label{homflypt counts} 
Let $K \subset S^3$ be a link, and $L_K \subset X$ the associated Lagrangian in the resolved
conifold.  Let $P_{K, \lambda}(a, q)$ be the $\lambda$-colored HOMFLYPT polynomial of $K$. 
Let $Z \in \sk(L_K)((Q^{1/2}))$ be the skein-valued open Gromov-Witten invariant.  For appropriate
choice of 4-chain,
$$Z|_{z=q^{1/2} - q^{-1/2}} = \sum_\lambda P_{K, \lambda}(Q^{1/2}, q) \cdot W_\lambda $$
\end{theorem}

\section{Totally isolated annuli} \label{isolated}

Let us fix some terminology.  Recall that for the perturbation scheme in \cite{bare} we always take almost 
complex structures which are standard in some neighborhood of the Lagrangians, and moreover only perturb
the Cauchy-Riemann equation in the complement of some smaller neighborhood of the Lagrangians. 

\begin{definition} \label{model} 
Let $\bT$ be the solid torus, which we identify as $S^1 \times \R^2$. 
Consider its cotangent bundle, $T^* \bT = T^* S^1 \times T^* \R^2$.   
Inside $T^* \R^2$, consider the zero section $Z$ and the the conormal at zero, $N$. 
Inside $T^*S_1 = S^1 \times \R$, we write $S_\epsilon$ for the circle over 
$\epsilon \in \R$; these are all Lagrangian circles but only $S_0$ is exact.  

Consider $L_1 = Z \times S_0$ and $L_2 = N \times S_\epsilon$.  For a standard complex
structure, there will be a unique holomorphic annulus stretching between them, namely 
$S^1 \times [0, \epsilon] \times \mathbf{0} \subset S^1 \times \R \times T^* \R^2$.  
We term this the {\em model annulus}. 
\end{definition}

\begin{definition} \label{standard annuli} 
Suppose given some (possibly disconnected) Lagrangian $L$.  Recall from \cite{SOB, bare} that we always take complex structures
standard within some $\epsilon$ neighborhood of the Lagrangian.  We say a holomorphic annulus ending on $L$ 
is {\em standard} if it is contained within this $\epsilon$ neighborhood, and some neighborhood of the annulus
can be identified with the model annulus.  
\end{definition}

\begin{remark}
Note that a small deformation of any two Lagrangians that meet cleanly along a non null-homologous knot bounds a standard annulus after a small shift.
\end{remark}

\begin{definition} \label{totally isolated} 
We say  a standard annulus is {\em totally isolated} if in addition no other holomorphic
curves with boundary on $L$ enter its standard neighborhood (aside from multiple covers of the annulus). 
\end{definition}

Here we will restrict ourselves to the totally isolated case, but we expect more sophisticated arguments (such as
in \cite{Ionel-Parker-GV}) will allow
us to treat more general annuli.  The `totally isolated' restriction allows us to avoid discussing any properties of the perturbation
scheme, beyond locality.   In particular we have: 

\begin{lemma} \label{universal}
There is some element $\Psi \in \sk(\bT) \otimes \sk(\bT)$ such that the total contribution of all multiple covers of a 
totally isolated annulus to the Gromov-Witten invariant is always given by the cabling of its boundary by $\Psi$
(up to monomial change of framing variables). 
\end{lemma} 
\begin{proof}
The perturbation scheme of \cite{bare} is local in the sense that we are free to independently perturb maps 
with disjoint images.  In particular we may always transplant some given perturbation of the model annulus 
and all its multiple covers to any occurence of a totally isolated annulus.  
\end{proof}

\section{Constraints from the cotangent bundle of the solid torus}

Here we study the model geometry of Definition \ref{model}.  We regard $T^* S^1 \times T^* \R^2$ 
as a Liouville manifold by taking the radial Liouville form on $T^* \R^2$, i.e., rounding corners in the $T^{\ast}\R^{2}$-factor.  We write $\Lambda_1$
and $\Lambda_2$ for the Legendrians at infinity of the Lagrangians $L_1, L_2$.  

We are interested in boundaries of moduli of 1-parameter families of curves asymptotic to 
chords on $\Lambda_1 \cup \Lambda_2$.  Boundaries coming from interior breaking we cancel
using the skein relations as in \cite{SOB, bare, unknot}; the remainder are the SFT breakings which correspond to two-level
curves
\cite{BEHWZ}.  These are pairs of: a rigid curve in the symplectization of the contact boundary, and a rigid curve in the
original manifold, with matching asymptotic Reeb chords.  We study such curves with a single positive puncture at Reeb chord connecting a component of $\Lambda$ to itself, and we will characterize which such curves may appear in the symplectization. 

For each Reeb chord of the Legendrian Hopf link, the Legendrian $\Lambda_1 \cup \Lambda_2$ has a Bott family, 
canonically parameterized by the zero section $S^1 \subset T^* S^1$. We resolve this degeneracy perturbing the contact form using 
a Morse function
on $S^1$ with one minimum.  
After perturbation,  
rigid holomorphic curves on $\Lambda_1 \cup \Lambda_2$ sit near this point, 
and understanding them reduces to the corresponding problem for 
the Legendrian Hopf link. More precisely:

\begin{lemma} \label{bott}
Minimal index Reeb chords for $\Lambda_1 \cup \Lambda_2$ 
are in natural bijection with minimal index chords for the Legendrian Hopf link.   
Rigid (up to translation) curves with one positive puncture are arbitrarily close to configurations of the following form: a curve of the Hopf link over the minimum in the $S^{1}$-family of self chords, and in case there are negative punctures, continue as flow lines to the minimum over trivial strips in the $S^{1}$-families of these.
\end{lemma}

\begin{proof}
		When resolving the Bott degeneracy
		we must glue holomorphic curves and Morse flow lines to the minima in the $S^1$ Bott-families.  This is straightforward in the case under consideration, compare e.g., \cite{EENS} for similar gluing results. Another approach is to work with flow trees throughout, and refer to \cite{E} for the relation to holomorphic curves. This was carried out for Legendrian tori constructed from general Legendrian isotopies of Legendrian links in \cite[Theorem 1.1]{EK}. In that setting, the case under consideration here, corresponds to the trivial isotopy of the standard Legendrian Hopf link.
\end{proof}

We now recall the chords and curves for the Hopf link.  

\begin{figure}[h!]
	\centering
	\begin{tikzpicture}[scale=1]

		\tikzset{->-/.style={decoration={ markings,
					mark=at position #1 with {\arrow{>}}},postaction={decorate}}}
		
		\draw [blue, thick=1.5] (-1.5,1) to[in=90,out=190] (-2.3,0);
		\draw [blue, thick=1.5] (-1.5,1) to[in=135,out=10] (-0.3,0.3);         
		\draw [blue, thick=1.5]  (-1.5,-1)to[in=270,out=170] (-2.3,0);
		\draw [blue, thick=1.5] (-1.5,-1)to[in=225,out=350] (-0.3,-0.3)  ;
		
		\draw [blue, thick=1.5] (1.5,1) to[in=45,out=170] (0.3,0.3);         
		
		\draw [blue, thick=1.5]  (1.5,-1)to[in=270,out=10] (2.3,0);
		\draw [blue, thick=1.5] (1.5,1) to[in=140,out=350] (1.9,0.8);
		\draw [blue, thick=1.5] (2.05,0.65) to[in=90,out=320] (2.3,0);

		\draw [blue, thick=1.5] (1.5,-1)to[in=315,out=190] (0.3,-0.3)  ;
		
		\draw [blue, thick=1.5] (-0.3,0.3) to (0.3,-0.3);
		\draw [blue, thick=1.5] (-0.3,-0.3) to (-0.1,-0.1);
		\draw [blue, thick=1.5] (0.1,0.1) to (0.3,0.3);
		
		\draw [red, thick=1.5] (2.5,1) to[in=90,out=190] (1.7,0);
		
		\draw [red, thick=1.5] (2.5,-1) to[in=320,out=170] (2.1,-0.8);
		\draw [red, thick=1.5] (1.95,-0.65) to[in=270,out=140] (1.7,0);

		\draw [red, thick=1.5] (2.5,1) to[in=135,out=10] (3.7,0.3);         
		\draw [red, thick=1.5] (2.5,-1)to[in=225,out=350] (3.7,-0.3)  ;
		
		\draw [red, thick=1.5] (5.5,1) to[in=90,out=350] (6.3,0);
		\draw [red, thick=1.5] (5.5,1) to[in=45,out=170] (4.3,0.3);         
		\draw [red, thick=1.5]  (6.3,0)to[in=10,out=270] (5.5,-1);
		\draw [red, thick=1.5] (5.5,-1)to[in=315,out=190] (4.3,-0.3)  ;
		
		\draw [red, thick=1.5] (3.7,0.3) to (4.3,-0.3);
		\draw [red, thick=1.5] (3.7,-0.3) to (3.9,-0.1);
		\draw [red, thick=1.5] (4.1,0.1) to (4.3,0.3);
		
		\node at (0,0.3) {\footnotesize{$c_1$}}; 
		\node at (2,1.2) {\footnotesize{$m_{12}$}}; 
		\node at (2,-1.2) {\footnotesize{$m_{21}$}}; 
		\node at (4,0.3) {\footnotesize{$c_2$}}; 
		
	\end{tikzpicture}
	\caption{The Lagrangian projection of the Legendrian Hopf link}
	\label{hopf}
\end{figure}
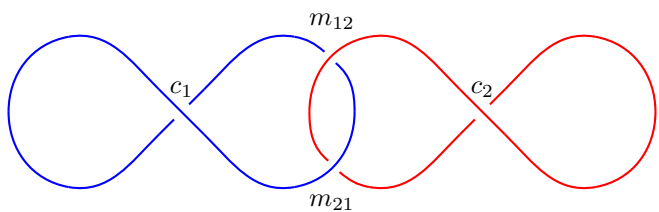


\begin{lemma}\label{1dimHopfdisk}
The rigid curves in the symplectization with one puncture at $c_1$ are: 
no-negative-puncture disks $D_1$ going to the left and $D_1'$ going to the right in Figure \ref{hopf}; 
and there is a two-negative-puncture disk $T_1$ which follows along $D_1'$ until arriving at $m_{12}$,
then changes to the other component, travels the short arc to $m_{21}$, then returns to again follow
$D_1'$. 

Similarly, for $c_2$ there is $D_2$ going to the right, $D_2'$ going to the left, and a two-negative-puncture
$T_2$ which again changes components at $m_{12}$ and $m_{21}$. 
\end{lemma}

The torus $\Lambda_1$ comes with 4 points where it meets the Reeb chords; similarly for $\Lambda_2$.  
We fix a capping path for the self-chord, but do not fix any for the mixed chord.  

\begin{lemma}
We use capping paths given by the boundaries $\partial D_1', \partial D_2'$. 
With this choice, $\partial D_1, \partial D_2$ become the meridians of their
respective tori, 
$\partial D_1', \partial D_2'$ become trivial, and $\partial T_1$ and $\partial T_2$ are isotopic (rel boundary). 
\end{lemma}

\begin{proof}
	Clear from Lemmas \ref{bott} and \ref{1dimHopfdisk}.
\end{proof}

\begin{proposition} \label{torus recursion} 
The count $\Psi$ of all bounded curves in the interior is annihilated by 
$$(P_{1,0} - \bigcirc) \otimes a_2 - a_1 \otimes (P_{1,0} - \bigcirc)$$
\end{proposition} 

\begin{proof}
Let us write $\mathcal{M}(c_1)$ and $\mathcal{M}(c_2)$ for the moduli spaces of holomorphic curves in 
the cotangent bundle of the solid torus with boundaries on $L_1 \cup L_2$ and with one positive
puncture, asymptotic, respectively, to $c_1$ or $c_2$.  These moduli spaces are one
dimensional.  They have boundaries of two kinds, coming from boundary degenerations in the interior,
and SFT degenerations at infinity.  Degenerations of the first kind are cancelled by working
in the skein (and appropriately weighting curves by Euler characteristic of the domain and 4-chain intersections as 
explained in \cite{SOB}).  The sum of all SFT degenerations must therefore vanish when evaluated in the
skein.  Such degenerations are two-level rigid curves.  There are two possibilities: either 
a disk $D_i$ or $D_i'$ at infinity, plus bounded curves in the interior; or  $T_1$ or $T_2$, plus some curves in the interior with two positive punctures asymptotic 
to $m_{12}$ and $m_{21}$.  

Let $\Psi$ denote the count of all bounded curves in the interior.  It takes value in 
$\sk^+(L_1 \cup L_2)$.  We write $L_1''$ and $L_2''$ for these Lagrangians with the two
points where the mixed Reeb chords enter and leave.  We write $\Psi'' \in \sk(L_1'' \cup L_2'')$ for the 
count of all holomorphic curves with two punctures, asymptotic to these two chords.  

We obtain two equations in $\sk(L_1 \cup L_2)$ from
the boundaries of $\mathcal{M}(c_1)$ and $\mathcal{M}(c_2)$.  
Up to unknown framing factors (signed monomials in framing variables on every term) these are: 
$$(\partial D_1 + \partial D_1') \circ \Psi + (\partial T_1) \circ'' \Psi'' = 0 $$
$$(\partial D_2 + \partial D_2') \circ \Psi + (\partial T_2) \circ'' \Psi'' = 0 $$

Here the $\circ''$ is  gluing of a $T^2 \times [0,1]$ with two
marked points on $T^2 \times 1$ to a solid torus with two marked points on the outside. 
We have an isotopy $\partial T_1 \sim \partial T_2$, so subtracting (multiples by some framing factor)
we may cancel the $\Psi''$ term and obtain: 
$$((\gamma_1 P_{1,0} + \gamma_1' \bigcirc) \otimes 1 + 1 \otimes (\gamma_2 P_{1,0} + \gamma_2' \bigcirc))\Psi = 0$$

Here the $\gamma_i$ are framing factors we now write explicitly. 
This operator preserves the natural grading by number of boxes in each factor of $\sk^+(\bT)$. 
We determine the framing factors from the first terms of 
$\Psi = 1 + \gamma \cdot W_\square \otimes W_\square + \cdots$.  From the zeroeth term, we see that 
$\gamma_1 = -\gamma_1'$ and $\gamma_2 = -\gamma_2'$.  Now we study the first term: 

$$\bigg(\gamma_1(P_{1,0} - \bigcirc) \otimes 1 + 1 \otimes \gamma_2 (P_{1,0} - \bigcirc) \bigg)  (W_\square \otimes W_\square) = 0$$

Using Equation (\ref{eigenvalue}), we see $a_2 \gamma_1 + a_1 \gamma_2 = 0$.  
\end{proof}  

\begin{corollary} \label{diagonal}
$\Psi = \sum_\lambda n_\lambda \cdot W_\lambda \otimes W_\lambda$, for some $n_\lambda(a_1, a_2, q)$. 
\end{corollary}
\begin{proof}
Expand $\Psi = \sum_{\lambda, \mu} n_{\lambda, \mu}(a_1, a_2, q) \cdot W_\lambda \otimes W_\mu$.  
Applying the operator of Proposition \ref{torus recursion} and using Equation (\ref{eigenvalue}), 
we find
$$ 0 =  a_1 a_2 (q^{1/2} - q^{-1/2}) \sum_{\lambda, \mu}  n_{\lambda, \mu}(a_1, a_2, q) \cdot (c_\lambda(q) - c_\mu(q)) \cdot
W_\lambda \otimes W_\mu$$
Since $c_\lambda(q) - c_\mu(q)$ vanishes only for $\lambda = \mu$, the $n_{\lambda, \mu}$ must
vanish whenever $\lambda \ne \mu$. 
\end{proof} 

\begin{remark}
Given the value of $\Psi$ we ultimately compute, in fact
$(\partial T_1) \circ'' \Psi''$ must be a rather nontrivial quantity.  
\end{remark}

\section{Proof of Theorems \ref{maintheorem} and \ref{homflypt counts}} 

\begin{proof}[Proof of Theorem \ref{maintheorem}.]
We must show that the  coefficients $n_\lambda(a_1, a_2, q)$ from Corollary \ref{diagonal} are in fact all 
just $\gamma^{|\lambda|}$, where $\gamma$ is some signed monomial in the framing variables $a_1, a_2$.  
To do so we study in $T^*S^3$ the union of the zero section and the conormal to an unknot, shifted off
as in \cite{SOB}.  Let us take the zero section as $L_1$ and the conormal as $L_2$. 

We write $\langle \cdot \rangle_{S^3}$ for the isomorphism from $\sk(S^3)$ to an appropriately localized
polynomial ring in $a_1, q$.  By Corollary \ref{diagonal} and Lemma \ref{universal}, the total count
of bounded curves is: 
$$\sum_\lambda n_\lambda(a_1, a_2, q) \langle W_\lambda \rangle_{S^3} \cdot W_\lambda$$
On the other hand, we computed this total count in \cite{unknot}, and showed it was one of the following two things, 
depending on the orientation of $S^3$: 
$$\sum_\lambda \gamma^{|\lambda|} \langle W_\lambda \rangle_{S^3} \cdot W_\lambda \qquad \qquad 
\sum_\lambda \gamma^{|\lambda|} \langle W_{\lambda'} \rangle_{S^3} \cdot W_\lambda $$

Here $\lambda'$ is the conjugate partition to $\lambda$.  In \cite{unknot} these results appeared via the formulae:
\begin{equation} \label{dimension} \langle W_\lambda \rangle_{S^3} = \prod_{\square \in \lambda} \frac{a q^{c(\square)/2} - a^{-1} q^{-c(\square)/2}}{q^{h(\square)/2} - q^{-h(\square)/2}}
\qquad \qquad \langle W_{\lambda'} \rangle_{S^3} = \prod_{\square \in \lambda} \frac{a q^{-c(\square)/2} - a^{-1} q^{c(\square)/2}}{q^{h(\square)/2} - q^{-h(\square)/2}}
\end{equation}

Thus we see that either $n_\lambda(a_1, a_2, q) = \gamma^{|\lambda|}$ or 
$n_\lambda(a_1, a_2, q) = \gamma^{|\lambda|}  \langle W_{\lambda'} \rangle_{S^3}  / \langle W_{\lambda} \rangle_{S^3}$. 
In the second case, we would have $n_\lambda(a_1, a_2, q) = a_2^m f_\lambda(a_1, q)$, where $f_\lambda$ is not
a monomial.  
But we also could have closed off $L_2$ to an $S^3$, leaving $L_1$ alone.   From this we would see 
either $n_\lambda(a_1, a_2, q) = \gamma^{|\lambda|}$ or that $n_\lambda(a_1, a_2, q) = a_1^{m'} g_\lambda(a_2, q)$, 
where $g_\lambda$ is not a monomial.
The only consistent possibility is $n_\lambda(a_1, a_2, q) = \gamma^{|\lambda|}$.
\end{proof}  

\begin{remark}
In \cite{unknot}, we found two possibilities because we used only the recursion relation
coming from the knot conormal, which cannot know the orientation of $S^3$.  Here however, both components
of the Lagrangian are oriented: we orient their longitudes along the direction traversed 
by the boundary of the annulus, and their meridians along $\partial D_1$ or $\partial D_2$.  Having already oriented
both components, there is no choice remaining when we close one off into a sphere.  In principle one could
follow carefully the 4-chain conventions to see which possibility from \cite{unknot} arises; but in the above argument
we found a trick to avoid doing so. 
\end{remark}

\begin{proof}[Proof of Theorem \ref{homflypt counts}]
This follows immediately from combining Theorem \ref{maintheorem} (via Lemma \ref{universal}) with
our previous results \cite[Theorem 6.6, 6.7, 7.3; Corollary 7.4]{SOB}, where the ``appropriate choice of 4-chain'' is given. 
To appeal to the results of \cite{SOB} in the present context, 
where multiple covers are not excluded for topological reasons, requires an adequate perturbation
scheme.  The perturbations must also have the locality properties demanded in Lemma \ref{universal}. 
We construct  such a scheme in \cite{bare}.  
\end{proof}  

\newpage

\bibliographystyle{hplain}
\bibliography{skeinrefs}

\end{document}